\documentclass[11pt,hyp]{nyjm}

\usepackage{hyperref}
\usepackage{graphicx,amssymb,upref}

\theoremstyle{plain}
\newtheorem{theorem}{Theorem}[section]
\newtheorem{lemma}[theorem]{Lemma}
\newtheorem{proposition}[theorem]{Proposition}
\newtheorem{corollary}[theorem]{Corollary}

\theoremstyle{definition}
\newtheorem{definition}[theorem]{Definition}
\newtheorem{question}[theorem]{Question}

\theoremstyle{remark}
\newtheorem{remark}[theorem]{Remark}

\newtheorem*{namedtheorem}{\namedtheoremname}
\newcommand{\namedtheoremname}{}
\newenvironment{named}[1]{%
  \renewcommand{\namedtheoremname}{#1}%
  \begin{namedtheorem}}
 {\end{namedtheorem}}

\let\<\langle
\let\>\rangle
\newcommand{\doi}[1]{doi:\,\href{http://dx.doi.org/#1}{#1}}

\title[Satellite knots and positive braids with full twists]{Satellite knots that cannot be represented by positive braids with full twists}
\author{Thiago de Paiva}
\address[]{School of Mathematics, Monash University, VIC 3800, Australia}
\email[]{thhiagodepaiva@gmail.com}

\keywords{satellite knots; positive braids; Lorenz knots; braid index; full twists.}
\subjclass[2010]{Primary 57K10; Secondary 57K35.}

\begin{document}

\begin{abstract}
A positive braid with at least one full twist is known to be a minimal braid, i.e, it achieves the braid index for its closure. In this paper we find knots that are the closure of positive minimal braids that cannot be represented by positive braids with full twists. More precisely, we show that some satellite knots with companions and patterns given as the closure of positive braids cannot be represented as the closure of positive braids with full twists. As a consequence, we find infinitely many satellite knots with companions and patterns being Lorenz knots that are not Lorenz knots.
\end{abstract}
\maketitle
\tableofcontents



\section{Introduction}


The braid index of a link is defined to be the smallest number of strings needed to represent it as the closure of a braid. The braid index is one of the simplest link invariants and is still not fully understood.

Franks and Williams proved that positive braids with at least one positive full twist are minimal braids \cite[Corollary 2.4]{Franks}. In this paper, we study when knots given by the closure of positive braids cannot reach the braid index with positive braids with at least one positive full twist. More precisely, we study when satellite knots with companions and patterns given by the closure of positive braids cannot be represented by positive braids with at least one positive full twist.

Denote by $\sigma_1, \dots, \sigma_{r_n-1}$ the standard generators of the braid group $B_{r_n}$.


Consider a positive braid $B$ with $b$ strands. We say that $B$ has $k>0$ full twists if $(\sigma_1\dots\sigma_{b-1})^{-kb}B$ is a positive braid.

To obtain a satellite knot we start with a knot  $P$ inside the solid torus $S^1\times D^2$ that is not isotopic into a ball nor to $S^1\times 0$ in $S^1\times D^2$.  Next, we consider a non-trivial knot $C$ and a homeomorphism $f\colon S^1\times D^2 \to N(C)$ that takes $S^1\times D^2$ to the tubular neighbourhood $N(C)$ of $C$.
Then, the image $K = f(P)$ is the satellite knot with \emph{pattern} $P$ and \emph{companion} $C$.
In this paper, we consider that the patterns of the satellite knots are represented by braids, which are naturally embedded in a solid torus formed by removing its braid axis from the 3-sphere.
However, there are infinitely many homeomorphisms taking $S^1\times D^2$ to $N(C)$. Hence we need to give $f$ an additional condition to make $K$ well-defined. With this in mind, the convention is that $f$ must send the standard longitude $S^1\times\{1\}$ of $S^1\times D^2$ to the standard longitude of $N(C)$. In other words, $f$ is not allowed to introduce additional twisting. In the case where the pattern is given by a braid with \( b \) strands inside the solid torus \( S^1 \times D^2 \), and the companion \( C \) is represented by a positive braid with \( c \) crossings. After applying the homeomorphism that takes the pattern \( P \) along the companion \( C \), while preserving the standard longitudes, to obtain the satellite knot \( K \), the knot inside the companion acquires \( c \) negative full twists on the \( b \) strands, since each crossing of \( C \) induces a positive full twist on the knot inside the companion when applying the inverse homeomorphism that maps \( N(C) \) to \( S^1 \times D^2 \). This homeomorphism is constrained so as not to introduce additional twisting. This concept is illustrated in \cite[Figure 1]{twofulltwists}.

In this paper we find knots represented by positive minimal braids that cannot be represented by any positive braids with at least one full twists as minimal braids. Our main theorem is the following.


\begin{theorem}\label{main}
Consider $c_1, d, b \geq 2$ and $c_2, k\geq0$. Let $C$ be a knot represented by a positive minimal braid with $c_1$ crossings and braid index equal to $d$. Let $P$ be a knot represented by a positive braid $\beta = (\sigma_1\dots\sigma_{b-1})^{c_1b+k}\beta'$ with $\beta'$ a positive braid with $b$ strands and $c_2$ crossings such that
$$c_1b^2 + k(b-1)+c_2\leq(db+1)(db-1)+1.$$
Then, the satellite knot $K$ with pattern $P$ and companion $C$ cannot be represented as a positive braid with at least one positive full twist.
\end{theorem}


As a corollary of this theorem, we find infinitely many satellite knots that are not Lorenz knots but have companions and patterns that are Lorenz knots, as shown below.

The meteorologist Edward Lorenz has described a three-dimensional system of ordinary differential equations to predict weather patterns \cite{Lorenz}. The Lorenz equations also arise in simplified models for lasers, dynamos, thermosyphons, electric circuits, chemical reactions, etc. And for some choices of parameters, it has knotted periodic orbits, which are called \emph{Lorenz links}.
Lorenz links have been given special attention, especially Lorenz knots which are \emph{satellites}.


There is a conjecture about satellite Lorenz knots, called the Lorenz satellite conjecture, which claims that satellite Lorenz knots have companions and patterns equivalent to Lorenz knots (see \cite[Conjecture 1.2]{de2021satellites}).

The converse of the Lorenz satellite conjecture is a question that naturally arises. See \cite[Question 1.4]{de2021satellites}. More precisely:
\begin{question}\label{mainquestion}
Let $K_1$ and $K_2$ be Lorenz knots. Is the satellite knot with pattern $K_1$ and companion $K_2$ also a Lorenz knot?
\end{question}

For positive answers to Question~\ref{mainquestion}, the first results addressing this question gave some partial positive answers. The first was given by Birman and Williams in 1980 when they proved that
if $K$ is a Lorenz knot with crossing number $c$ and $a, b’$ are arbitrary coprime positive integers, then
the satellite knot with companion the Lorenz knot $K$ and pattern the torus knot $T(a, b’+ ac)$ is a Lorenz
knot \cite[Theorem 6.2]{Knottedperiodicorbits}. However, as discussed in Remark~\ref{remark}, we present counterexamples to this result, which follow as a consequence of the main theorem in this paper. Further supporting a positive answer, de Paiva and Purcell proved that the answer to Question~\ref{mainquestion} is yes for large families of Lorenz knots \cite[Theorem 1.3, Theorem 6.1, and Theorem 6.2]{de2021satellites}. For negative answers to Question~\ref{mainquestion}, it is known that the $(2,3)$-cable of the $(2,3)$-torus knot is not a Lorenz knot since it can't be represented by a positive braid. But, as a corollary of Theorem~\ref{main}, we find infinitely many satellite knots which are not Lorenz knots but have Lorenz knots as companions and patterns:

\begin{corollary}\label{maintheorem}
There are infinitely many satellite knots
that are not Lorenz knots that have both companions and patterns that are Lorenz knots.
\end{corollary}

Therefore, we give infinitely many negative answers to Question~\ref{mainquestion}.

The structure of the paper is as follows. In the first section we prove Theorem~\ref{main}. In the second section we give a different proof, first proved by Birman and Williams in \cite[Theorem 5.1]{Knottedperiodicorbits}, for the fact that every T-link has a positive braid with at least one positive full twist, Theorem~\ref{fulltwist}. And so, we obtain Corollary~\ref{maintheorem}.

\subsection{Acknowledgment}I thank the anonymous referee for their detailed and helpful comments, which significantly improved the exposition and clarity of this paper.

\section{Satellite knots and number of crossings}

In this section we prove Theorem~\ref{main}.

The next lemma is a well-known result in the literature of positive braid theory. We include its proof here for completeness.

\begin{lemma}\label{genus}
Let $B_1, B_2$ be two positive braids with the same number of strands representing the same knot $K$. Then, $B_1$ and $B_2$ have the same number of crossings.
\end{lemma}

\begin{proof}
The genus of $K$ is given by the formula
$$\dfrac{c_i-b_i+1}{2},$$
where $c_i$, $b_i$ is the number of crossings, strands, respectively, of $B_i$ with $i = 1$ or $2$ \cite{Genus}. As the genus is a knot invariant and $B_1$ and $B_2$ have the same number of strands, it follows that $B_1$ and $B_2$ have the same number of crossings.
\end{proof}

\begin{lemma}\label{crossings}
Suppose that $B$ is a positive braid with $p$ strands and at least one full twist on $p$ strands with closure representing a knot. Then, the number of crossings of $B$ is more than $$p(p-1) + p - 2.$$
\end{lemma}

\begin{proof}
By contradiction, suppose $B$ has less than or equal to $p(p-1) + p - 2$ crossings.

The braid $B$ has at least $p(p-1)$ crossings because $B$ has at least one positive full twist. Hence the braid $B$ is formed by adding at most $p - 2$ crossings to the braid given by a positive full twist on $p$ strands.


We apply a negative full twist on the braid axis of $B$ to remove $p(p-1)$ crossings from $B$. Then, we obtain a braid $B'$ with $p$ strands and at most $p-2$ crossings. As full twists induce the trivial permutation on the endpoints of the strands, the closures of $B$  and $B'$ have the same number of link components.

A positive braid with $p$ strands and at most $p - 2$ crossings cannot use all $p - 1$ generators. Some $\sigma_i$ is not used, so the strands $\lbrace 1, \dots , i\rbrace$ never cross the strands $\lbrace i + 1, \dots , p\rbrace$. Hence, these two sets
are invariant by the permutation induced by the braid, so the closures of $B$  and $B'$ have at least two components.
\end{proof}


\begin{lemma}\label{isotopy}
Consider $B = \sigma_{i_1}\dots \sigma_{i_{p-1}}$ a positive braid with $p$ strands and $p-1$ crossings representing a knot. Then, the braid $B$ is conjugate to $\sigma_{1}\dots \sigma_{p-1}$.
\end{lemma}

\begin{proof}
Each crossing $\sigma_{j}$ with $j\in \lbrace 1, \dots, p-1 \rbrace$ appears at least once in $B$ otherwise the closure of $B$ would be a link with more than one component, which is not possible. Furthermore, each crossing $\sigma_{j}$ with $j\in \lbrace 1, \dots, p-1 \rbrace$ appears exactly once in $B$ as $B$ has $p-1$ crossings.

If $\sigma_{i_1}\neq \sigma_{1}$, then we push all crossings in $B$ before $\sigma_{1}$ once around the braid closure so that $B$ is transformed into the braid $B = \sigma_{1}\sigma_{i'_2}\dots \sigma_{i'_{p-1}}$ by a conjugation. If $\sigma_{i'_2}\neq \sigma_{2}$, then all crossings in $B_1$ between $\sigma_{1}$ and $\sigma_{2}$ commute with $\sigma_{1}$, hence they can be moved to the beginning of the braid, and then by a conjugation they can be moved to the end, obtaining a braid $B_2 = \sigma_{1} \sigma_{2} \sigma_{i''_3}\dots \sigma_{i''_{p-1}}$.
At every step, the crossings between $\sigma_{1}\dots \sigma_{i}$ and $\sigma_{i+1}$ commute with $\sigma_{1}\dots \sigma_{i}$, hence they can be moved to the beginning, and then to the end by a conjugation, yielding a positive braid which starts with $\sigma_{1}\dots \sigma_{i+1}$.
Continue until we obtain the braid $B_{p-1} = \sigma_{1}\dots \sigma_{p-1}$.
\end{proof}

The braid $B = \sigma_{i_1}\dots \sigma_{i_{p-1}}$ in Lemma~\ref{isotopy} actually  represents the trivial knot.

\begin{lemma}\label{lemma1}
Suppose that $B$ is a positive braid with $p$ strands, one full twist on $p$ strands, and with closure representing a knot. Then,
if $B$ has $$(p+1)(p-1)$$ crossings, its closure is the $(p, p+1)$-torus knot.
\end{lemma}

\begin{proof}
If $F =  (\sigma_{1}\dots \sigma_{p-1})^{p}$ is the full twist on $p$ strands, then $B = FB'$ where $B'$ is a positive braid with $p - 1$ crossings. Since the permutation induced by $F$ is trivial and $B$ represents a knot, the closure of $B'$ is also a knot, hence, by Lemma~\ref{isotopy}, $B'$ is conjugate to $\sigma_{1}\dots \sigma_{p-1}$. Since $F$ is central in the braid group
on $p$ strands, $B$ is conjugate to $F\sigma_{1}\dots \sigma_{p-1}$, that is, to $(\sigma_{1}\dots \sigma_{p-1})^{p+1}$, which is the $(p, p+1)$-torus knot.
\end{proof}

\begin{lemma}\label{lemma2}
Suppose that $B$ is a positive braid with $p$ strands and one full twist on $p$ strands.
If $B$ has $$(p+1)(p-1) +1$$ crossings, then its closure cannot be a knot.
\end{lemma}


\begin{proof}
The number $(p+1)(p - 1)+1$ has the same parity as $p$.
So, the permutation induced by $B$ has the same parity as $p$, so it cannot be a cycle of
length $p$ (which has the same parity as $p - 1$). Hence the closure of $B$ cannot be a knot.
\end{proof}



\begin{lemma}\label{lemma3}
Suppose that $B$ is a positive braid with $p$ strands and at least one full twist on $p$ strands. If the closure of $B$ represents a satellite knot, then its number of crossings is more than $$(p+1)(p-1) +1.$$
\end{lemma}

\begin{proof}
It follows from Lemma~\ref{crossings} that $B$ has more than $p(p-1) + p - 2$ crossings. Since $B$ has closure representing a satellite knot and a satellite knot cannot be a torus knot by Thurston~\cite{Thurston}, it follows from Lemma~\ref{lemma1} that $B$ has more than $(p+1)(p-1)$ crossings. Finally, $B$ has more than $(p+1)(p-1)+1$ crossings due to Lemma~\ref{lemma2}.
\end{proof}

Theorem 1.4 of \cite{unexpected} illustrates satellite knots with positive braids with $p$ strands, one full twist on $p$ strands, and $(p+1)(p-1) +2$ crossings.

\begin{named}{Theorem~\ref{main}}
Consider $c_1, d, b \geq 2$ and $c_2, k\geq0$.
Let $C$ be a knot represented by a positive minimal braid with $c_1$ crossings and braid index equal to $d$. Let $P$ be a knot represented by a positive braid $\beta = (\sigma_1\dots\sigma_{b-1})^{c_1b+k}\beta'$ with $\beta'$ a positive braid with $b$ strands and $c_2$ crossings such that
$$c_1b^2 + k(b-1)+c_2\leq(db+1)(db-1)+1.$$
Then, the satellite knot $K$ with pattern $P$ and companion $C$ cannot be represented as a positive braid with at least one positive full twist.
\end{named}

\begin{proof}
The knot $K$ has braid index equal to $db$ by \cite[Theorem 1]{Williams}.

Suppose that $K$ can be represented by a positive braid $B$ with $n$ strands and at least one full twist on $n$ strands.

By \cite[Corollary 2.4]{Franks}, the number $n$ is equal to the braid index of $B$. As $B$ represents $K$, we conclude that $n$ is equal to $db$.

The minimal positive braid for $C$ induces a minimal positive braid $B'$ for $K$ by Williams \cite{Williams} or Birman and Menasco \cite[Corollary 3]{MR1286930}. Furthermore, since this minimal positive braid for $C$ has $c_1$ crossings, then after applying the homeomorphism that tangles the knot $P$ along the knot $C$ preserving
the standard longitudes to obtain $K$, the knot inside the companion receives $c_1$ negative full twists on the $b$ strands. So, $B'$ has $$c_1b^2 + k(b-1)+c_2$$ crossings.

By Lemma~\ref{genus}, $B$ has the same amount of crossings as $B'$. However, this contradicts Lemma~\ref{lemma3} as $c_1b^2 + k(b-1)+c_2$ is less than or equal to $(db+1)(db-1)+1$.
\end{proof}

The family of T-links are defined as follows:
for $2\leq r_1< \dots < r_k$, and all $s_i>0$, the T-link $T((r_1,s_1), \dots, (r_k,s_k))$ is defined to be the closure of the following braid on $r_k$ strands, which we call the standard braid:
\[ (\sigma_1\sigma_2\dots\sigma_{r_1-1})^{s_1}(\sigma_1\sigma_2\dots\sigma_{r_2-1})^{s_2}\dots(\sigma_1\sigma_2\dots\sigma_{r_k-1})^{s_k}.\]
In particular, we conclude that torus knots are T-links from these representations for T-links (or \cite[Theorem 6.1]{Knottedperiodicorbits}). Birman and Kofman showed that T-links are equivalent to Lorenz links \cite[Theorem 1]{newtwis}.

\begin{corollary}\label{corollary1}
Consider $c_1, d_1, \dots, c_n, d_n, a, b, k$ positive integers with  $1< c_1 < \dots < c_n <b$ and $a>1$ such that
the T-link $$P = T( (c_1, d_1), \dots, (c_n, d_n), (b, (a-1)(a+1)b + k))$$ is a knot and
$$b^2 - kb + k - \sum_{i=1}^{n} d_i(c_i-1) \geq 0.$$
Then, the satellite knot $K$ with pattern the T-knot $P$ and companion the torus knot $C = T(a, a+1)$ cannot be represented as a positive braid with full twists.
\end{corollary}

\begin{proof}
The torus knot $C = T(a, a+1)$ has a minimal positive braid with $a$ strands, one full twist on $a$ strands, and $(a-1)(a+1)$ crossings.

The T-knot $P = T( (c_1, d_1), \dots, (c_n, d_n), (b, (a-1)(a+1)b + k))$ is given by the positive braid
\[ (\sigma_1\dots\sigma_{c_1-1})^{d_1}\dots(\sigma_1\dots\sigma_{c_n-1})^{d_n}(\sigma_1\dots\sigma_{b-1})^{(a-1)(a+1)b + k}.\]
Consider the braid $B' = (\sigma_1\dots\sigma_{c_1-1})^{d_1}\dots(\sigma_1\dots\sigma_{c_n-1})^{d_n}$ with
$$(c_1-1)d_1+\dots +(c_n-1)d_n$$crossings. The knot $P$ is the closure of the braid $B = (\sigma_1\dots\sigma_{b-1})^{(a-1)(a+1)b + k}B'$.

The inequality of the hypothesis implies that
$$- b^2 + kb - k + \sum_{i=1}^{n} d_i(c_i-1)\leq 0,$$
which is equivalent to
$$a^2b^2-ab^2+ab^2 -b^2 + kb - k +  \sum_{i=1}^{n} d_i(c_i-1)\leq a^2b^2 -ab+ab- 1+1,$$ or simply
$$(a-1)(a+1)b^2 + k(b-1)+ \sum_{i=1}^{n} d_i(c_i-1)\leq (ab+1)(ab-1)+1.$$

Therefore, the satellite knot $K$ with pattern $P$ and companion $C$ cannot be represented as a positive braid with at least one positive full twist by Theorem~\ref{main}.
\end{proof}








\section{Minimal braids for T-links}

In \cite[Proposition 5.6]{Knottedperiodicorbits} Birman and Williams proved that every non-trivial T-link is the closure of a positive braid with at least one positive full twist. In this section we provide a different proof for this fact.



\begin{definition}
The braid $\sigma_{i\rightarrow j}$ is defined as $\sigma_i\sigma_{i+1}\dots\sigma_{j-1}$ if $i < j$, as $\sigma_{i-1}\sigma_{i-2}\dots\sigma_{j}$ if $i > j$, and as the trivial braid if $i = j$.
\end{definition}

Consider a braid $B$ with $p$ strands. We say that the
span the \emph{span} of $B$, denoted by span($B$), is contained in $[1, q]$ if the crossings of $B$ only involve the strands in the interval $[1, q]$ with $q\leq p$. In other words, $B$ only involves the generators $\sigma_{1},\dots, \sigma_{t}$ and their inverses for some $t \leq q - 1$.

\begin{proposition}\label{proposition10}
Consider $p, q$ positive integers with $p>q>1$. Let $K$ be a link which is the closure of a positive braid $B$ with $p$ strands of the form
$$B'(\sigma_{1\rightarrow p})^q$$ with $B'$ a positive braid on $p$ strands and span$(B')$ contained in $[1, q]$. Then, the link $K$ is the closure of a positive braid with at least one positive full twist.
\end{proposition}

\begin{proof}
We start by deforming the sub-braid $(\sigma_{1\rightarrow p})^q$ of $B$ to the braid $(\sigma_{1\rightarrow q})^p$ using the isotopy of the second drawing of Figure~\ref{C2}, where this isotopy can be described as a rotation around a diagonal of the square that gives the torus in which the $(p, q)$-torus knot lies.
After this, the sub-braid $B'$ is sent to lie horizontally on the meridional (horizontal) lines of the braid
$(\sigma_{1\rightarrow q})^p$. Since $B'$ has less than or equal to $q$ strands, we can push $B'$ up so that it lies
on the longitudinal (vertical) lines of the braid $(\sigma_{1\rightarrow q})^p$, as illustrated by the third drawing of Figure~\ref{C2}.
We now have a positive braid with q strands, with the sub-braid $(\sigma_{1\rightarrow q})^p$ at the bottom. Moreover, this braid has at least one full twist as $p>q$.
\end{proof}

\begin{figure}
\includegraphics[scale=0.12]{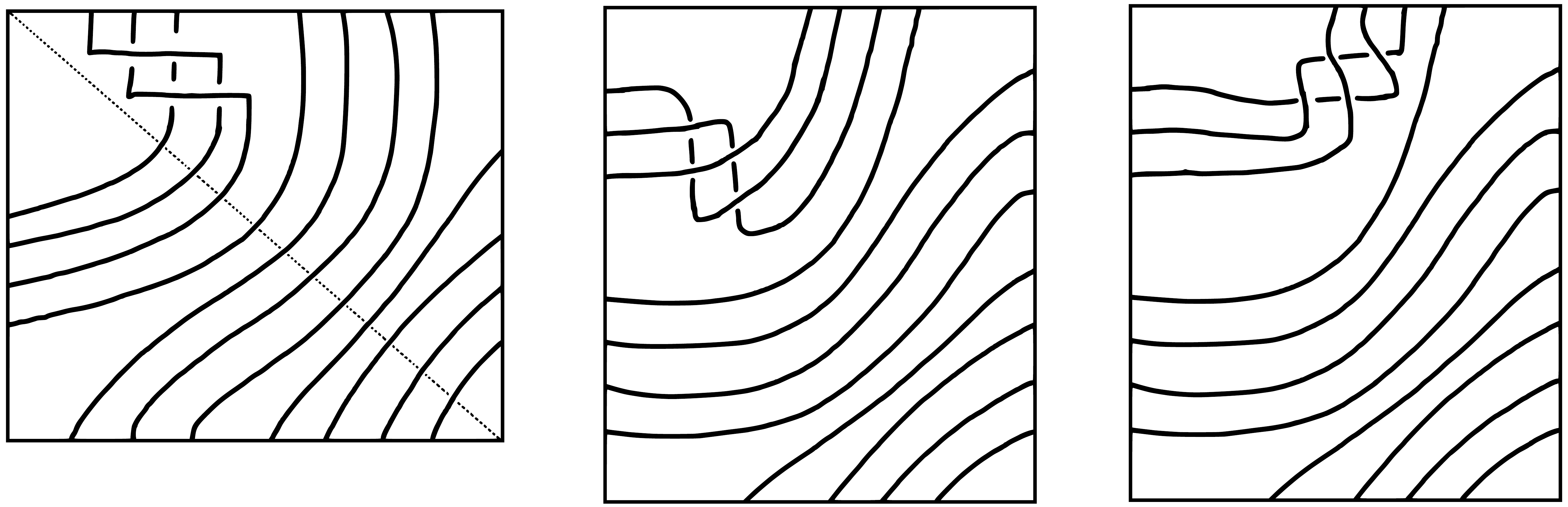}
\caption{The second drawing is derived from the first by a 180-degree rotation around the dashed diagonal of the square, which defines the torus on which the \( (7, 4) \)-torus knot lies. The third drawing is obtained from the second by pushing the braid \( (\sigma_2 \sigma_1)^2 \) upward so that it aligns with the longitudinal lines of the \( (4, 7) \)-torus knot.}
\label{C2}
\end{figure}

The next proposition generalizes the isopoty of \cite[Lemma 2.1]{de2022torus}, which was based on \cite[Figure 6]{Positively}.

\begin{proposition}\label{isopote}
If $0 < q \leq r < p$ and $B$ is a braid on $p$ strands with span$(B) \subset [1, r]$, then the $p$-braid $B(\sigma_{1\rightarrow p})^q$ has the same closure as the $r$-braid $$B(\sigma_{q\rightarrow 1})^{p-r} (\sigma_{1\rightarrow r})^q.$$
\end{proposition}

\begin{proof}
The case $q=1$ immediately follows by shrinking (or applying some destabilization moves to) the last $p-r$ strands of $B(\sigma_{1\rightarrow p})^q$. Hence we assume that $q>1$.

The $(r-q+1)$-st strand anticlockwise goes around the braid closure to pass through the sub-braid $(\sigma_{1\rightarrow p})^q$ and finally ends in the $(r+1)$-st strand as $q\leq r$.
The $(r-q+1)$-st strand of the braid $B(\sigma_{1\rightarrow p})^q$ is connected to the $(r+1)$-st strand by an under strand that goes one time around the braid closure as illustrated by the blue strand in Figure~\ref{T1}.
We push this under strand down and shrink it to reduce one strand from the last braid, as shown in Figure~\ref{T1}. After that, this strand becomes an under strand going from the $r$-th to the $(r - q + 1)$-st position yielding the braid
$$(\sigma_{r\rightarrow r-q+1})B(\sigma_{1\rightarrow p-1})^q.$$

\begin{figure}
\includegraphics[scale=0.7]{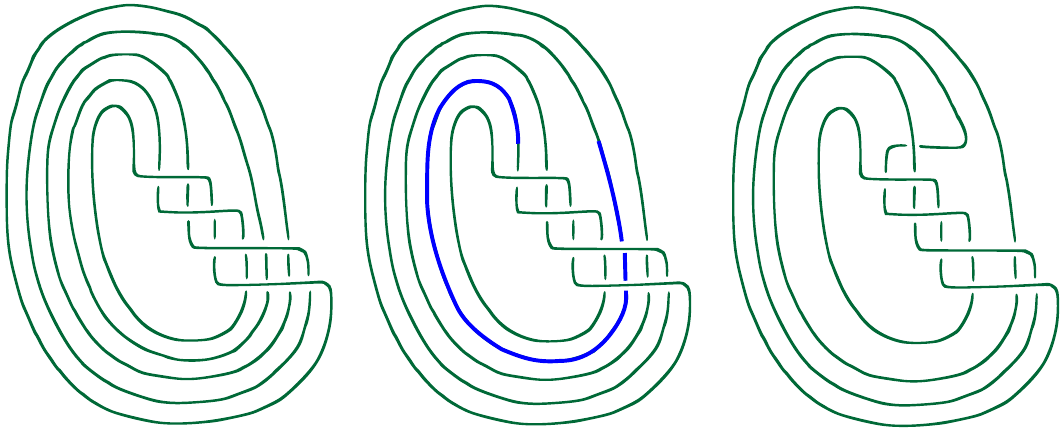}
\caption{This series of drawings illustrates how to reduce one strand from the closure of the leftmost braid, as used in the proof of Proposition~\ref{isopote}, where \( p = 5 \), \( q = 2 \), and \( r = 3 \). In the second drawing, the second strand at the top of the braid moves anticlockwise around the braid closure once, forming the blue strand. Since the blue strand is an under strand, it can be pushed down and shrunk to obtain the rightmost braid.}
\label{T1}
\end{figure}

We see that this isotopy doesn't change the sub-braid $B$. So, it happens in the complement of the braid axis of $B$.

As the sub-braid $(\sigma_{r\rightarrow r-q+1})B$ of the last braid is a braid on the leftmost $r$ strands, the $(r-q+1)$-st strand is still connected to the $(r+1)$-st strand by an under strand that goes one time around the braid closure.
So, we can apply the last isopoty again to obtain the braid
$$(\sigma_{r\rightarrow r-q+1})^2B(\sigma_{1\rightarrow p-2})^q.$$

We continue applying this procedure and then after doing it $p-r$ times, we obtain the braid on $r$
strands
$$(\sigma_{r\rightarrow r-q+1})^{p-r}B(\sigma_{1\rightarrow r})^q,$$
which is conjugate to $$B(\sigma_{1\rightarrow r})^q(\sigma_{r\rightarrow r-q+1})^{p-r}.$$ Now, since the span of the last factor is contained within the interval $[r - q + 1, r]$ (the last q strands), we can push the last factor upward  to obtain the braid
$$B(\sigma_{q\rightarrow 1})^{p-r}(\sigma_{1\rightarrow r})^q. \qedhere$$
\end{proof}


Now, given a $p$-braid $B$ with span$(B) \subset [1, p - 1]$, denote by shift$(B)$ the braid obtained by replacing each crossing $\sigma_i$ of $B$ with $\sigma_{i+1}$. Hence span(shift$(B)) \subset [2, p]$. Notice that
$$(\sigma_{1}\dots \sigma_{p-1})B = \text{shift}(B)(\sigma_{1}\dots \sigma_{p-1}).$$

\begin{proposition}\label{secondcase'}
If $$B = (\sigma_{1\rightarrow r_1})^{s_1}\dots (\sigma_{1\rightarrow r_{n-1}})^{s_{n-1}}B'(\sigma_{1\rightarrow r_{n}})^{s_{n}}$$
is an $r_n$-braid, where $n \geqslant 1$, $1 < r_1 < \dots < r_n$, $s_1 + \dots + s_n \geqslant 2$, and $B'$ is a positive braid on $r_n$ strands with span$(B') \subset [1, s_n]$, then the link $L$ given by the closure of $B$ can be represented by a positive braid with a full twist.
\end{proposition}

\begin{proof}
We prove this proposition by induction.

Consider $n = 1$. Then, $B = B'(\sigma_{1\rightarrow r_1})^{s_1}$ with $s_1 \geqslant 2$. If $s_1 \geqslant r_1$, then $B$ already has a full twist. If $2 \leqslant s_1 < r_1$, then the closure of $B$ can be represented by a positive braid with a full twist by Proposition~\ref{proposition10}.

Suppose $n > 1$ and that the result holds for all values smaller than $n$. If $s_n \geqslant r_n$, then $B$ already has a full twist. If $r_{n - 1} \leq s_n < r_n$, then $B = B_0(\sigma_{1\rightarrow r_n})^{s_n}$ with span$(B_0) \subset [1, s_n$] and
$s_n \geq 2$, so the desired braid exists by Proposition~\ref{proposition10}. Finally, suppose that $s_n < r_{n -                                    1} < r_n$.

\begin{figure}
\includegraphics[scale=1.3]{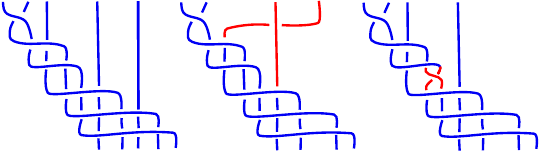}
\caption{The first braid represents the T-link $T((2, 1), (3, 2), (4, 1), (5, 2))$. We apply the isotopy of Proposition~\ref{isopote} to the first braid to obtain the second braid. Then, we push the first crossing of type $\sigma_{3}$ once counterclockwise to obtain the third braid.}
\label{T2}
\end{figure}

By Proposition~\ref{isopote}, the link L can be represented by the $r_{n-1}$-braid
$$B_1 = (\sigma_{1\rightarrow r_1})^{s_1}\dots (\sigma_{1\rightarrow r_{n-1}})^{s_{n-1}}B'(\sigma_{s_n\rightarrow 1})^{r_n-r_{n-1}}(\sigma_{1\rightarrow r_{n-1}})^{s_{n}} = $$
$$(\sigma_{1\rightarrow r_1})^{s_1}\dots (\sigma_{1\rightarrow r_{n-1}})^{s_{n-1}}B''(\sigma_{1\rightarrow r_{n-1}})^{s_{n}},$$
where $B'' = B'(\sigma_{s_n\rightarrow 1})^{r_n-r_{n-1}}$. Notice that span$(B'') \subset [1, s_n]$. Now we have
just two cases: If $s_{n - 1} + s_n \geqslant r_{n - 1}$, then $B_1$ has a full twist because
$$(\sigma_{1\rightarrow r_{n-1}})^{r_{n-1} - s_{n}}B''(\sigma_{1\rightarrow r_{n-1}})^{s_{n}} = \text{shift}^{r_{n-1} - s_{n}}(B'')(\sigma_{1\rightarrow r_{n-1}})^{r_{n-1}},$$
where the latter factor is a full twist. If $s_{n-1} + s_n < r_{n - 1}$, then $B_1$ can be written as
$$B_1 = (\sigma_{1\rightarrow r_1})^{s_1}\dots (\sigma_{1\rightarrow r_{n-2}})^{s_{n-2}}\text{shift}^{s_{n-1}}(B'')(\sigma_{1\rightarrow r_{n-1}})^{s_{n-1} + s_{n}},$$ as illustrated by Figure~\ref{T2}.
We notice that span(shift$^{s_{n-1}}(B'')) \subset [1, s_{n-1} + s_n]$, hence the result follows by induction hypothesis.
\end{proof}

\begin{corollary}\label{secondcase}
Consider $p>q>1$, $r_n\geq q$, $1< r_1 < \dots < r_n < p$, and $s_1, \dots, s_n>0$.
Then, the T-link $$T((r_1, s_1), \dots, (r_n, s_n), (p, q))$$ has a positive braid with at least one positive full twist.
\end{corollary}

\begin{proof}
It follows from Proposition~\ref{secondcase'} where $B'$ is the trivial braid.
\end{proof}


\begin{theorem}\label{fulltwist}
Every T-link $$T((r_1, s_1), \dots, (r_n, s_n), (p, q))$$ with $s_1 + \dots + s_n + q\geq 2$ is the closure of a positive braid with at least one positive full twist. Hence every non-trivial T-link is the closure of a positive braid with at least one positive full twist.
\end{theorem}

\begin{proof}
Consider $K$ the T-link $T((r_1, s_1), \dots, (r_n, s_n), (p, q))$, where $s_1 + \dots + s_n +q \geq 2 $.

If $q = 1$, then $K$ is equivalent to the T-link $T((r_1, s_1), \dots, (r_n, s_n +1))$ by Proposition~\ref{isopote}. So, we can assume that $q>1$.

If $q\geq p$, then the standard braid of $K$ is itself a positive braid with at least one positive full twist.

If $p>q\geq r_n$ (possibly $n = 0$), then $K$ has a positive braid with at least one positive full twist by Proposition~\ref{proposition10}

Finally, if $r_n>q$, then $K$ also has a positive braid with at least one positive full twist from Corollary~\ref{secondcase}.

If $s_1 + \dots + s_n +q  = 1 $, then $K$ is the trivial torus knot $T(p, 1)$. Therefore, every non-trivial T-link is the closure of a positive braid with at least one positive full twist.
\end{proof}

\begin{corollary}\label{corollary2}
Consider $c_1, d_1, \dots, c_n, d_n, a, b, k$ positive integers with $c_1<\dots < c_n<b$ and $a>1$ such that
the T-link $$P = T( (c_1, d_1), \dots, (c_n, d_n), (b, (a-1)(a+1)b + k))$$ is a knot and
$$b^2 - kb + k - \sum_{i=1}^{n} d_i(c_i-1) \geq 0.$$
Then, the satellite knot $K$ with pattern the T-knot $P$ and companion the torus knot $T(a, a+1)$ is not a T-knot.
\end{corollary}

\begin{proof}
It follows from Theorem~\ref{fulltwist} and Corolary~\ref{corollary1}.
\end{proof}

\begin{corollary}\label{maincorollary}
Consider $a, b$ integers greater than one.
The satellite knot with pattern the torus knot $T(b, (a-1)(a+1)b + 1)$ and companion the torus knot $T(a, a+1)$ is not a T-knot.
\end{corollary}

\begin{proof}
As $b^2-b+1\geq0$, the result follows from Corolary~\ref{corollary2}.
\end{proof}

Therefore, we have proved Corolary~\ref{maintheorem}.

\begin{remark}\label{remark}
From Corollary~\ref{maincorollary}, the satellite knot with pattern the torus knot $T(b, (a-1)(a+1)b + 1)$ and companion the torus knot $T(a, a+1)$ is not a T-knot. However, \cite[Theorem 6.2]{Knottedperiodicorbits} states that they are T-knots. Hence, we conclude that Corollary~\ref{maincorollary} provides counterexamples to \cite[Theorem 6.2]{Knottedperiodicorbits}. In particular, this suggests that the authors of \cite[Theorem 6.2]{Knottedperiodicorbits} did not fix the standard longitudes in this result. If we allow a homeomorphism (see the definition of satellite knots in the third paragraph in the introduction) to add as many full twists as we want on the pattern, then we can see that the satellite knot with pattern the torus knot $T(b, (a-1)(a+1)b + 1)$ and companion the torus knot $T(a, a+1)$ is a T-knot \cite[Theorem 6.2]{Knottedperiodicorbits}. However, when we do that, we break the convention that this homeomorphism must send the standard longitudes to the standard longitudes, and then the satellite operation is not well-defined.
\end{remark}



\end{document}